\documentclass[12pt]{article}
\usepackage[english]{babel}
\usepackage{graphicx}
\usepackage{hyperref}
\usepackage{amsmath, amssymb, amsthm, amsfonts, dsfont, color}
\usepackage{multicol}
\usepackage{array}

\newtheorem{Th}{Theorem}

\newtheorem{pro}{Proposition}

\newtheorem{cor}{Corollary}

\newtheorem{question}{Question}
\theoremstyle{definition}

\newtheorem{rem}{Remark}

\newcommand{\Tr}{\Delta}
\newcommand{\bbR}{\mathbb{R}}
\newcommand{\bbZ}{\mathbb{Z}}

\newcommand{\bbN}{\mathbb{N}}

\newcommand{\tH}{\widetilde{\mathrm{Homeo}}_+(S^1)}
\newcommand{\trot}{\mathop{\mathrm{rot}^{\sim}}}

\newcommand{\id}{\mathop{\mathrm{id}}}

\newcommand{\tz}{\mathcal Z}

\title{Self-similarity of Jankins-Neumann ziggurat}
\author{A.S. Gordenko}

\begin{document}
\maketitle


\section{Introduction}

In the theory of dynamical systems on the circle, there is the following very natural question: let $a,b\in \tH$ be lifts on the real line of two orientation-preserving circle homeomorphisms, and we know their rotation (or, more precisely, translation) numbers $\trot(a),\trot(b)\in\bbR$. What can be said about the translation number of their composition~$ab$?

Another, more general, form of the same question was studied in a work~\cite{JN} of Jankins and Neumann. It have had topological origins: the question of classification of 3-manifolds, admitting at the same time a Seifert fibration and a codimension one foliation transverse to it. By the moment of Jankins--Neumann's work, the only non-studied case was the one of a manifold fibered over a 2-sphere. Via the study of the corresponding holonomy maps, this have led them to the following question:
\begin{question}\label{q:JN}
Given $a_1, a_2, \dots, a_n \in [0,1]$, $n\ge 3$, when do there exist lifts $f_1, f_2, \dots, f_n$ of orientation-preserving homeomorphisms of the circle with $\trot(f_i)=a_i$, such that $f_1\dots f_n=id$?
\end{question}
They have suggested a conjectural answer to this question, also proving that it suffices to establish their conjecture for $n=3$ and that for $n=3$ their conjecture holds at least $99.9\%$ of the volume of the set. The set they have proposed for $n=3$ was later called the Jankins-Neumann ziggurat due to its stepwise nature (see~Fig.~\ref{f1}).


Their conjecture was proven by Naimi in~\cite{Nai}:
\begin{Th}[Naimi; conjecture of Jankins-Neumann]\label{zigJN}
The set defined in Question~\ref{q:JN} for $n=3$ is the union of parallelepipeds 
$$
[0;\frac{a}{m}]\times [0;\frac{m-a}{m}]\times [0;\frac{1}{m}]
$$
for all coprime $0<a<m$, and of their images under all the permutations of the coordinates.
\end{Th}

Later, Calegari and Walker attacked the question of the rotation number of the composition from the dynamical point of view. They have obtained an ``algorithmic'' description for analogous sets for any positive composition of two homeomorphisms with given rotation numbers:

\begin{figure}[h]
 \noindent\centering{
 \includegraphics[width=120mm]{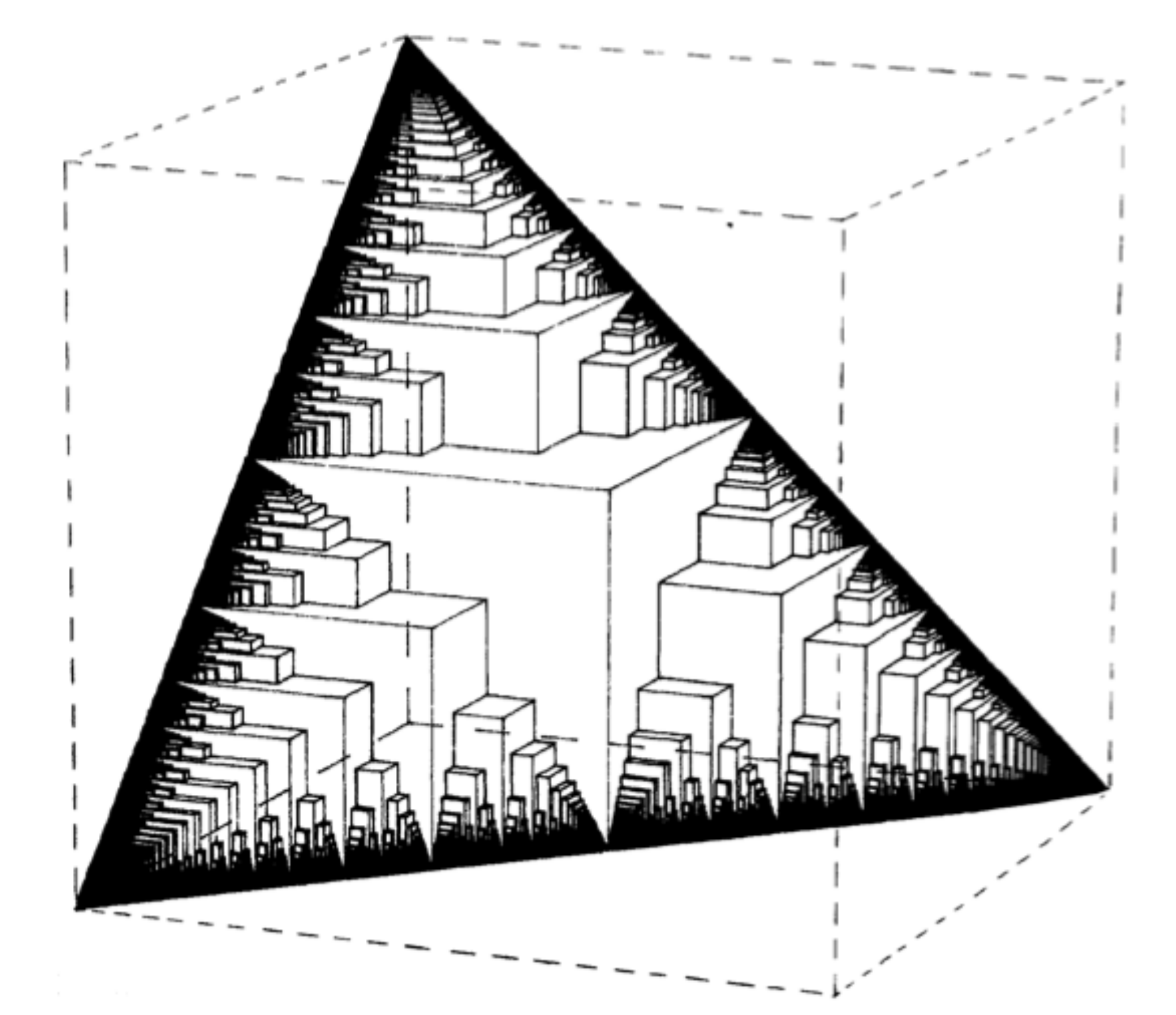} }
\caption{Jankins-Neumann ziggurat (picture credit: Jankins-Neumann~\cite{JN})}
\label{f1} \end{figure}

\begin{Th}[Calegari--Walker, \mbox{\cite{Caleg}}]\label{t:w}
For any word $w$ in the alphabet $a,b$, one has
$$
\{\trot(w(a,b)) \mid \trot(a)=x, \trot(b)=y\} = [r_{w}(x,y), R_{w}(x,y)]
$$
for certain functions $r_w, R_w:\bbR\to \bbR$, and one has $r_w(x,y)=-R_w(-x,-y)$. If the word $w$ is positive (i.e. contains no $a^{-1}$ or $b^{-1}$), there is an explicit algorithm to compute the functions $r_w$, $R_w$ at any rational point $(x,y)$.
\end{Th}

This theorem implies, in particular, that the ``ziggurat'' of possible rotation numbers of the composition is described by its upper boundary, the graph of the function~$R_{ab}(x,y)$. An immediate remark is that the function $R_{ab}(x,y)-x-y$ is $\bbZ^2$-periodic, thus to understand $R_{ab}(x,y)$ it suffices to study it on the unit square $[0,1)^2$. Moreover, starting from their dynamical approach, Calegari and Walker have obtained an explicit formula for $R_{ab}$ in terms, different from those of Jankins and Neumann:

\begin{Th}[{\itshape ab}~Theorem, \mbox{\cite{Caleg}}]\label{t:ab}
\begin{equation}\label{formula}
\forall x,y \quad R_{ab}(x, y) = \sup_{\frac{p_1}{q}\leq x, \frac{p_2}{q}\leq y} \frac{p_1 + p_2 +1}{q}.
\end{equation}
\end{Th}

The purpose of the present text is twofold. First, the Jankins--Neumann ziggurat clearly exhibits some fractal nature (see Fig.~\ref{f1},~\ref{ab}). We study the geometry of this ziggurat, and in particular show that it is indeed the case: that the set of its vertices is self-similar under some simple projective transformations. 

Secondly, as it was mentioned earlier, the theorem of Calegari and Walker and the Jankins--Neumann conjecture have the sets described in a different way. It is known that these two descriptions are equivalent (in particular, Calegari--Walker's Theorem~\ref{t:ab} gives an alternative proof of the Jankins--Neumann conjecture). However, we have found a very interesting proof of the passage between the two, and it seems that this way of proving the equivalence was not previously known, and is shorter than existing one. We present this passage in Sec.~\ref{s:equivalence}. Also, we discuss the corollaries of this comparison for the Calegari--Walker formula and the function~$R_{ab}$.

\section{Vertices and the self-similarity}\label{stat}

A first immediate remark is the passage between the definitions of the ziggurat by Jankins--Neumann and by Calegari--Walker: 


\begin{pro}\label{zigCW}
The triple $(x, y, z) \in [0;1)^3$ of translation numbers can be represented by $f_1, f_2, f_3 \in \tH$ such that $f_1f_2f_3=\id$ if and only if $z \in [0; R_{ab}(1 - x , 1 - y) - 1]$.
\end{pro}

In this section, we will be working with the Jankins--Neumann ziggurat, denoting it by~$\tilde{\tz}$. Naimi's Theorem~\ref{zigJN} then implies that


\begin{multline}\label{e:zigJN}
\tilde \tz = \bigcup_{\gcd(a,m) = 1} \Pi\left( \frac{a}{m},\frac{m-a}{m}, \frac{1}{m}\right) \cup \bigcup_{\gcd(a,m) = 1} \Pi\left( \frac{a}{m},\frac{1}{m}, \frac{m-a}{m}\right) \\
\cup \bigcup_{\gcd(a,m) = 1} \Pi\left( \frac{1}{m},\frac{a}{m}, \frac{m-a}{m}\right)
\end{multline}
where $\Pi(a,b,c) = [0;a]\times [0;b]\times [0;c]$. 


To state the self-similarity result, let us first introduce the notion: we call point $X$ a \emph{vertex} of a ziggurat $Z$, that is the union of parallelepipeds, if $X \in Z$ and there is no $\Pi(a,b,c) \subset Z$ such that $X \in \Pi(a,b,c)\setminus \{(a,b,c)\}$. In terms of $R_{ab}$, a \emph{vertex} is a point $(x,y,R_{ab}(x,y))\in [0,1)^3$, such that for any $x'\le x$ and $y'\le y$ inequality $x'+y'<x+y$ implies $R(x',y')<R(x,y)$.

Now, we can list the vertices of the Jankins--Neumann ziggurat: a direct corollary of the Naimi's theorem is 
\begin{pro}\label{p:list}
The vertices of the Jankins--Neumann ziggurat are points of the three families, that differ by the permutation of the coordinates:
\begin{itemize}
\item $\{(\frac{a}{m},\frac{m-a}{m},\frac{1}{m})\}$, with coprime $m>a>0$;
\item $\{(\frac{m-a}{m},\frac{1}{m},\frac{a}{m})\}$, with coprime $m>a>0$;
\item $\{(\frac{1}{m},\frac{m-a}{m},\frac{a}{m})\}$, with coprime $m>a>0$.
\end{itemize}
These three families lie respectively on the planes $x+y=1$, $x+z=1$ and $y+z=1$. Moreover, they lie respectively inside the triangles $ABD$, $ACD$ and $BCD$, where $A$, $B$ and $C$ are respectively the points at unit distance on the axes $Ox$, $Oy$ and $Oz$, and $D=(1/2,1/2,1/2)$ is the only common point of all the three families. 
\end{pro}

\begin{rem}\label{r:min}
Formally speaking, to ensure that all the points in the above list are the indeed the vertices, one should check that neither of the parallelepipeds listed in the Jankins--Neumann conjecture is contained in any other. Though, such a check is almost immediate (and we will do it in Sec.~\ref{s:equivalence}).
\end{rem}

\begin{cor}\label{c:list}
Translating Proposition~\ref{p:list} on the language of $R_{ab}$, we get for its (stepped) graph the families of vertices
\begin{itemize}
\item $\{(\frac{m-a}{m},\frac{a}{m},1+\frac{1}{m})\}$ with coprime $m>a>0$;
\item $\{(\frac{a}{m},1-\frac{1}{m},1+\frac{a}{m})\}$ with coprime $m>a>0$;
\item $\{(1-\frac{1}{m},\frac{a}{m},1+\frac{a}{m})\}$ with coprime $m>a>0$.
\end{itemize}
\end{cor}

Our {first} result, Theorem~\ref{ss} below, states that these points form a self-similar set. But before stating it, we would like to have additional geometric intuition on that set of vertices. Namely, at first glance it seems natural to decompose the ziggurat on Fig.~\ref{f1} into three parts ``near $A$'', ``near $B$'', ``near $C$''. However, as Proposition~\ref{p:list} shows, it is much more important to decompose the vertices into the three families listed in this proposition: the vertices that are on the triangles $ABD$, $ACD$ and $BCD$ respectively. 

{Also, we would like to deduce conclusions for the function $R_{ab}$. Thus it is interesting to consider the projection of the ziggurat on the $xy$ plane (marking the level surfaces and discontinuity lines): see Fig.~\ref{proj}. Marking only the vertices on this projection, we get Fig.~\ref{vert}.} Such a projection sends the vertices that correspond to the first family on the line $x+y=1$, and the second and the third families become separated by the diagonal $x=y$: the second comes below the diagonal, while the third one comes above. 


\begin{figure}[h] \begin{multicols}{2}
\hfill \includegraphics[width=70mm]{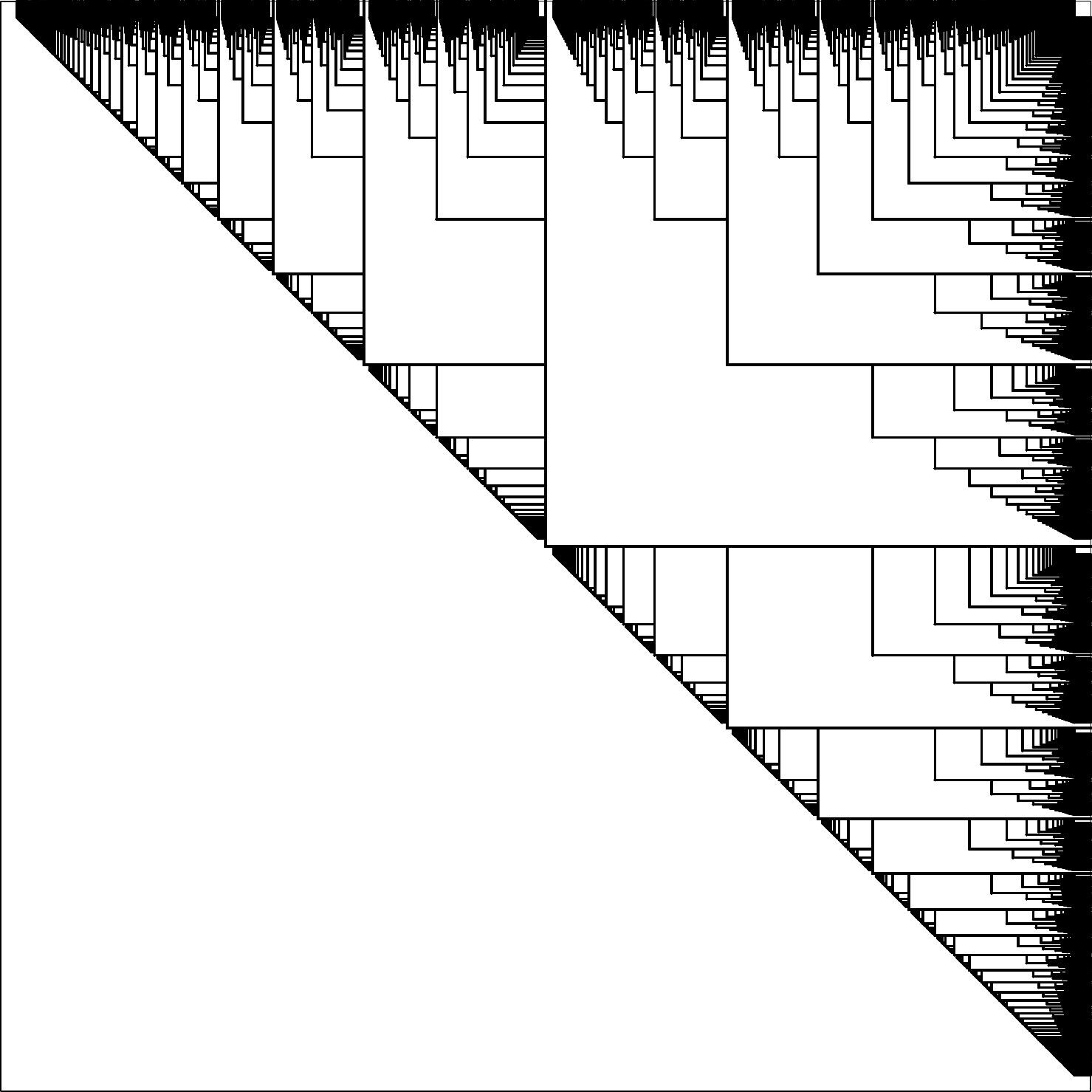} \hfill \caption{Graph of $R_{ab}$, seen from the top} \label{proj} \hfill \includegraphics[width=70mm]{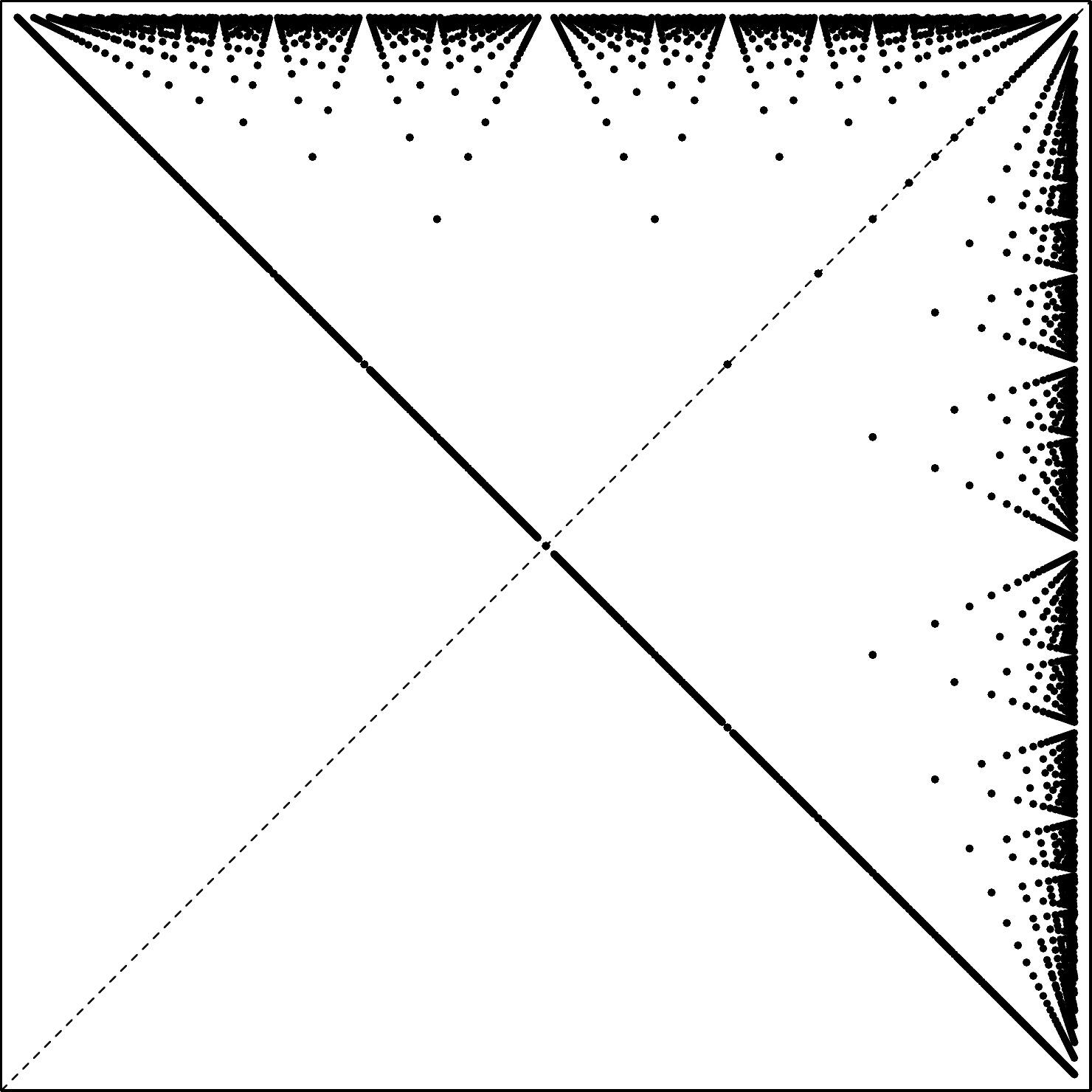} \hfill \caption{Projection of its vertices on the $xy$-plane, with the axis of symmetry marked} \label{vert}
\end{multicols} \end{figure}

Now, let $\Delta$ be the set that comes from the projection of vertices from the third family. As we have claimed before, this set is then self-similar:

\begin{figure}
\noindent\centering{
 \includegraphics[width=70mm]{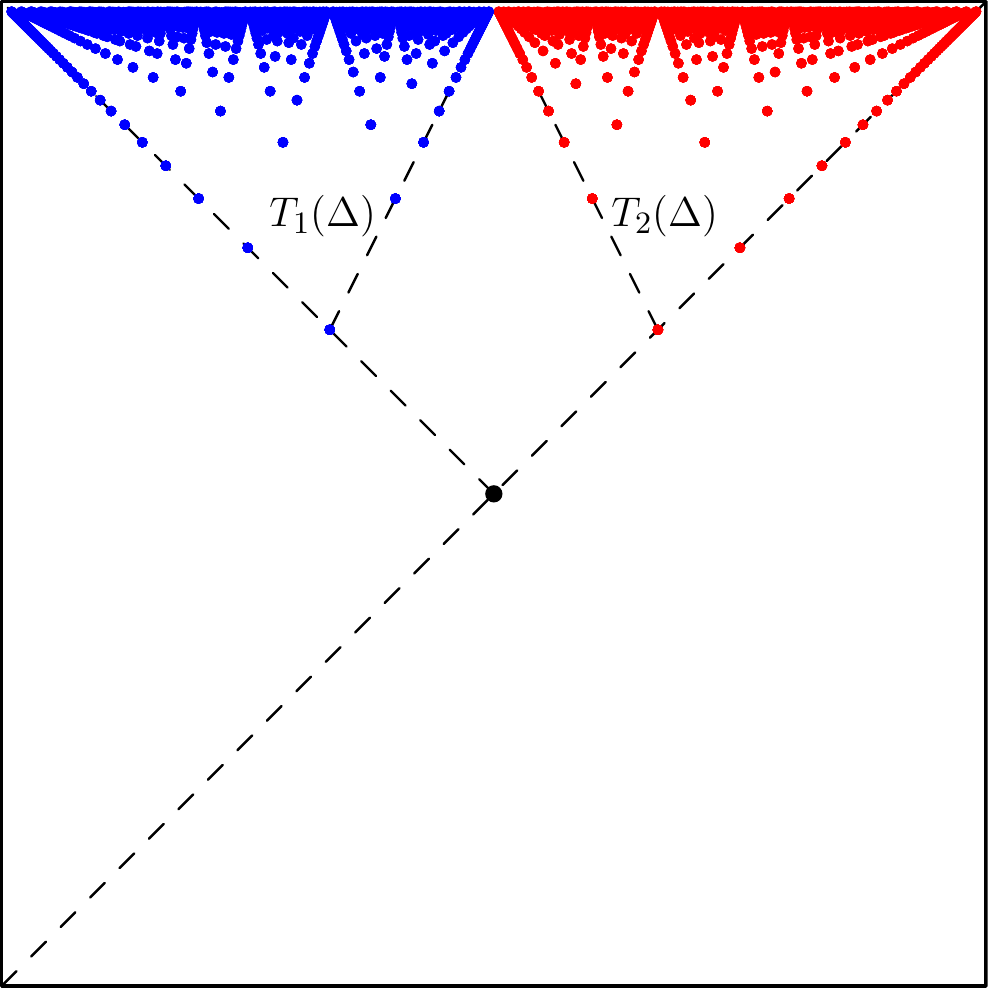} \hfill \caption{The set $\Delta$ and its self-similarity} \label{map}}
\end{figure}

\begin{Th}[Self-similarity]\label{ss}
The set $\Delta$ is self-similar with respect to two projective transformations 
$$
T_1(x,y)=\left( \frac{x}{1+x}, \frac{x+y}{1+x}\right), \quad T_2(x,y)=\left( \frac{1}{2-x}, \frac{1+y-x}{2-x}\right),
$$
namely, one has
$$
\Delta=T_1(\Delta) \sqcup T_2(\Delta) \sqcup \{\left(1/2,1/2\right)\}.
$$
\end{Th}

\begin{rem}\label{r:plane}
{As all the vertices of the third family lie on the plane $y+z=1$ of the triangle $ACD$, their projection to the $xy$ plane is an affine transformation. Hence, the set of the vertices of the third family is also self-similar under projective transformation that are lifts of $T_1$ and of $T_2$ on this plane. }
\end{rem}

\begin{rem}
{Note that the set $\Delta$ is symmetric with respect to the line $x=1/2$. It is an immediate observation if one considers the Jankins--Neumann ziggurat, coming from its full symmetry under the permutation of the coordinates. Though, it is much less clear if one studies the level curves of the function~$R_{ab}$: for any vertex and its symmetric image the projection of the ziggurat on the $xy$ plane collapses one of the sides of the corresponding parallelepiped, but the collapsed sides are different. This is why this symmetry does not appear on Fig.~\ref{proj}.}
\end{rem}

{Strangely enough, we did not find any direct way of establishing Theorem~\ref{ss} in the dynamical terms, that is, starting with the explicit formula~\eqref{formula} and without at first deducing the full list of the vertices (in a way it is done in Sec.~\ref{s:equivalence}). That is, one could imagine that having a vertex $(x,y)=(\frac{p_1}{q},\frac{p_2}{q})$ and considering its $T_1$ or $T_2$-image, formula~\eqref{formula} would allow to show that this image is also a vertex. Unfortunately, without stating the full list of vertices first such an argument does not seem to work. 
Instead, we obtain Theorem~\ref{ss} as a corollary of the Jankins--Neumann original description of the ziggurat (and the associated list of vertices from Proposition~1).}

Another conclusion for the function $R_{ab}$, following from the Jankins--Neumann description is the following
\begin{pro}\label{c:rational}
The function $R_{ab}$ takes only rational values; in formula~\eqref{formula}, for any $x$ and $y$ the supremum is a maximum. 
\end{pro}
This statement generalizes the Rationality Theorem of Calegari and Walker~\cite[Theorem 3.2]{Caleg}, giving the same rationality conclusion for the rational points~$(x,y)$.


It is also an interesting remark (though \emph{a posteriori} almost immediate to prove) that the projections of the vertices are indeed aligned along the lines that are ``visible'' on Fig.~\ref{vert}. The following theorem formalizes this statement; to state it, consider two families of lines. Namely, the ``green'' family of lines passing through the point $(0,1)$ and having slopes $(-1/m)$, $m=1,2,\dots$, and the ``red'' family of lines passing through the point $(1,1)$ and having slopes $1/k$, $k=1,2,\dots$. 

It is easy to check that the lines from these families are given by equations $y=\frac{2m-1}{m}-\frac{x}{m}$ and $y=\frac{k-1}{k}+\frac{x}{k}$ respectively. Let $\Delta'$ be the set of intersection points of lines of green family with lines of red family. Then, we have the following theorem, illustrated by~Fig.~\ref{f:lines}.
\begin{Th}[Alternative vertex set description]\label{lines}
$\Delta$ is the part of $\Delta'$ formed by the points with the least possible ordinate for given abscise:
$$
\Delta=\{(x,y)\in \Delta' \mid \forall y'<y \quad (x,y')\notin \Delta'\}.
$$
\end{Th}
\begin{rem}
Again, as in Remark~\ref{r:plane}, the construction of this theorem can be lifted on the $ACD$ plane containing the third family of vertices. Making such a lift, one notices that the vertical line starting from an intersection point corresponds to a vertical edge of the parallelepiped, starting from the corresponding vertex. The ``least possible ordinate'' rule then corresponds to the fact that the intersection points with non-least ordinate lift to the points that belong to the corresponding edge, and thus that are not vertices. 
\end{rem}

\begin{figure}
 \noindent\centering{
 \includegraphics[width=120mm]{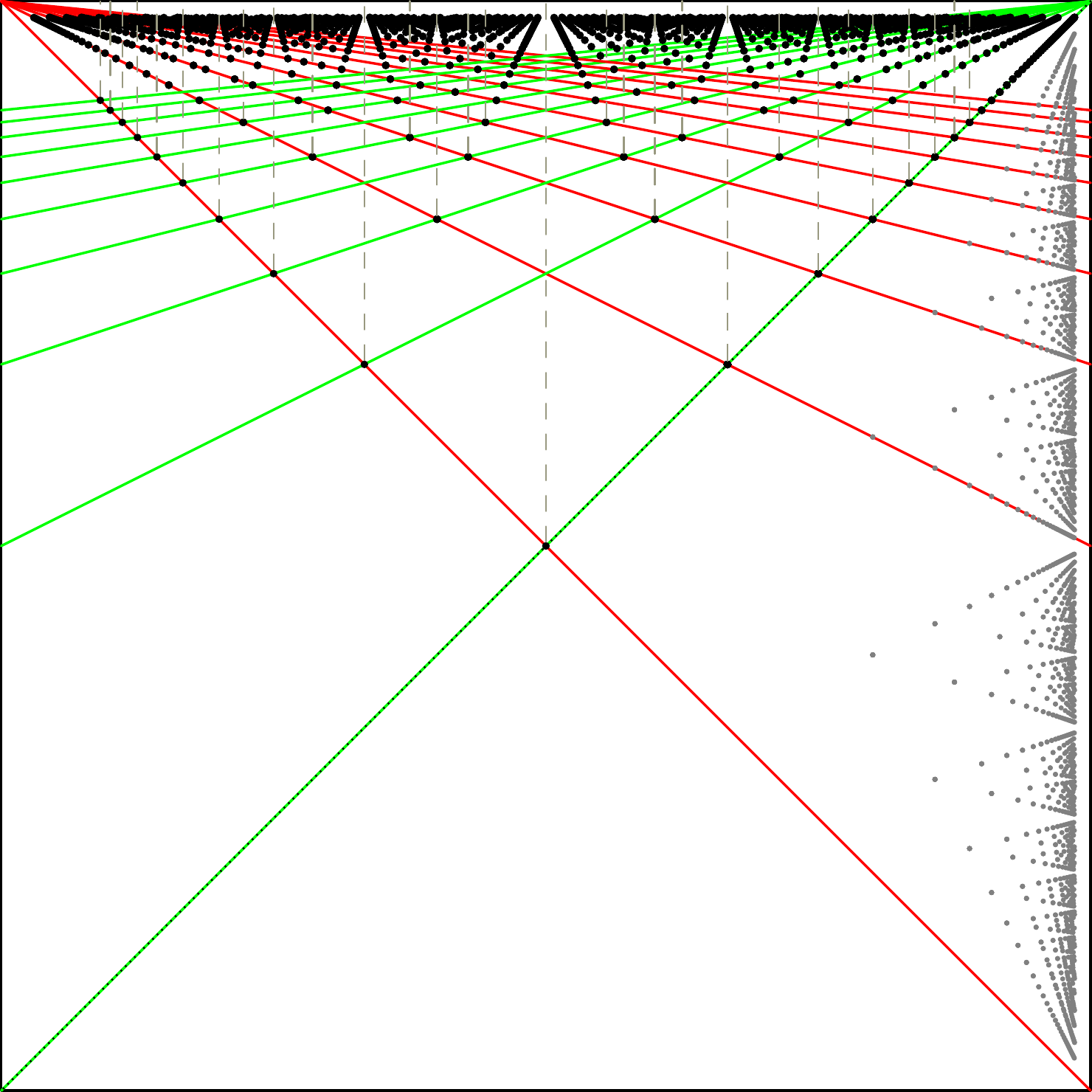} }
\caption{Lines listed in Theorem~\ref{lines}. Green and red families pass through the points (0,1) and (1,1) respectively. Bold black points correspond to the points of $\Delta$, the points of $X$ between diagonals $x=y$ and $x+y=1$ are marked by smaller grey points to illustrate additional aligning. Dashed vertical lines, starting in the points of~$\Delta$, illustrate the ``least possible ordinate'' condition.}
\label{f:lines} \end{figure}

\begin{figure}
 \noindent\centering{
 \includegraphics[width=120mm]{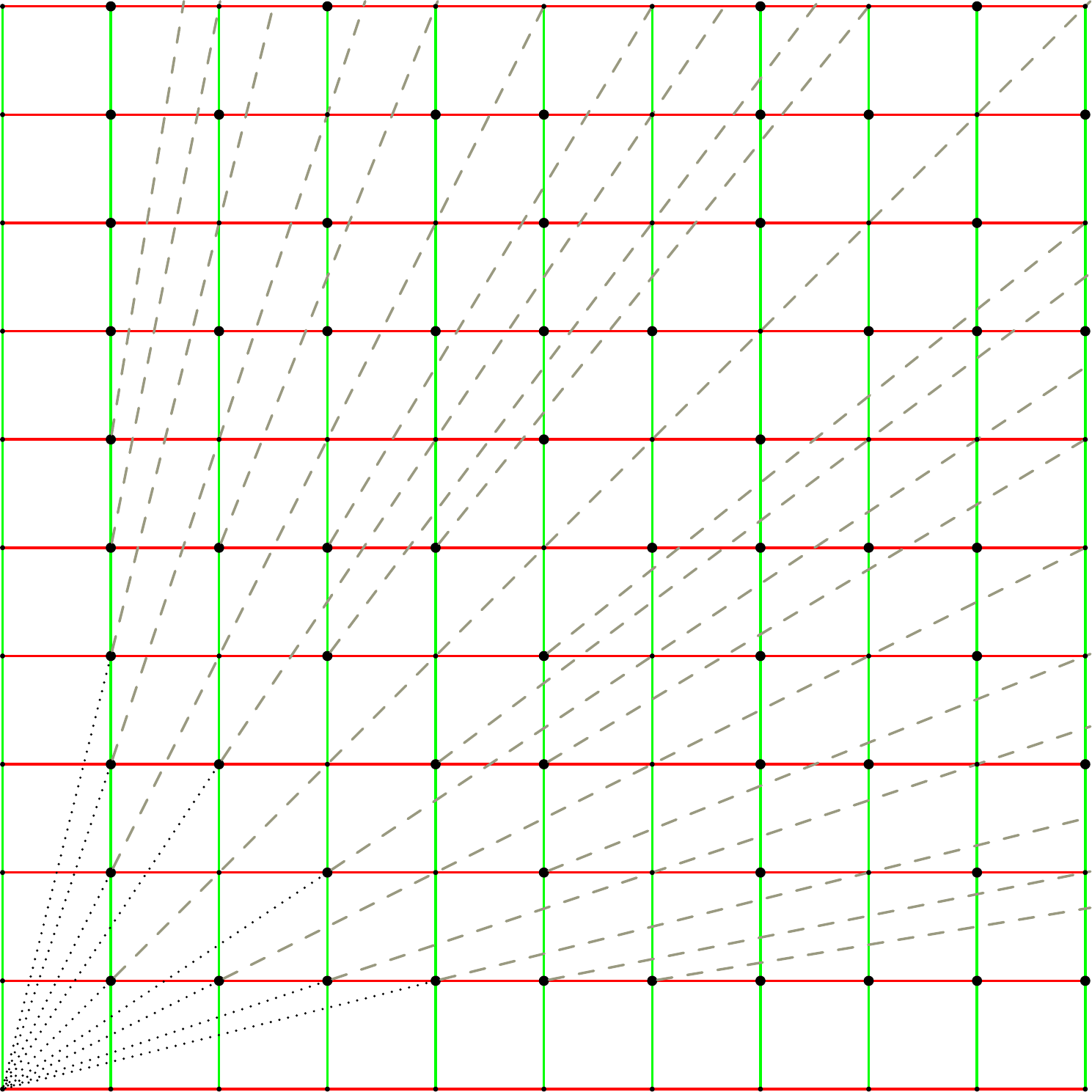} }
\caption{Fig. \ref{f:lines} after projective transformation $(x,y) \mapsto \left(\frac{x}{1 - y}, \frac{1-x}{1-y}\right)$. Some of dashed lines are not drawn, the extensions of some dashed lines are drawn dotted to illustrate passing through the origin.}
\label{f:strtn} \end{figure}

\begin{proof}[Proof of Theorems \ref{ss} and \ref{lines}]
First 
apply a projective transformation $Q$ that sends $(0,1)=[0:1:1]$ and $(1,1)=[1:1:1]$ to the points at infinity $V=[0:1:0]$ and $H=[1:0:0]$, corresponding to the  vertical and horizontal directions respectively, and that sends $V$ to the origin. It is easy to see that it is given by the formula 
$$Q: (x,y) \mapsto \left(\frac{x}{1 - y}, \frac{1-x}{1-y}\right).$$
or, equivalently, by
$$
Q:[x:y:t]\mapsto [x:t-x:t-y].
$$ 
The result of its application is shown on Fig.~\ref{f:strtn} (the reader perhaps will find this figure even more convincing than the formal arguments below).

It is easy to see that it sends a vertex $\left(\frac{r}{q}, \frac{q-1}{q} \right)$ (due to Proposition~\ref{p:list} and its Cor.~\ref{c:list}) the set $\Delta$ is formed by such points with coprime $0<r<q$) to the point $(r, q-r)$, and hence $Q(\Delta)$ is exactly the subset of $\bbN^2$ formed by points with coprime coordinates.

Next, an immediate check gives that the image of the line at infinity is the line $x=y$, the point $[a:1:0]$ (corresponding to the slope~$1/a$) being sent to the point~$(a,-a)$. Thus, the green family (as these are lines passing through $(0,1)$ and having slope $(-1/m)$) becomes the family of the vertical lines $x=m$, while the red family (as these are lines passing through $(0,1)$ and having slope~$(1/k)$) becomes the family of horizontal lines $y=k$. 

Hence the set $\Delta'$ is sent exactly to~$\bbN^2$. Finally, as the point $V$ is sent to the origin, the ``least possible ordinate'' condition  after the transformation~$Q$ becomes exactly the coprimality condition; this concludes the proof of Theorem~\ref{lines}.


To prove Theorem~\ref{ss}, note that in new coordinates the transformations $T_1$ and $T_2$ take form 
$$
\hat T_1(a,b) = (a, a+b); \, \text{ and } \,
\hat  T_2(a,b) = (a+b, b)
$$
and the self-similarity of $\Delta$ becomes obvious: in the new coordinates, it is the Euclid's algorithm!
\end{proof}

\begin{figure}[h] \begin{multicols}{2}
\hfill \includegraphics[width=70mm]{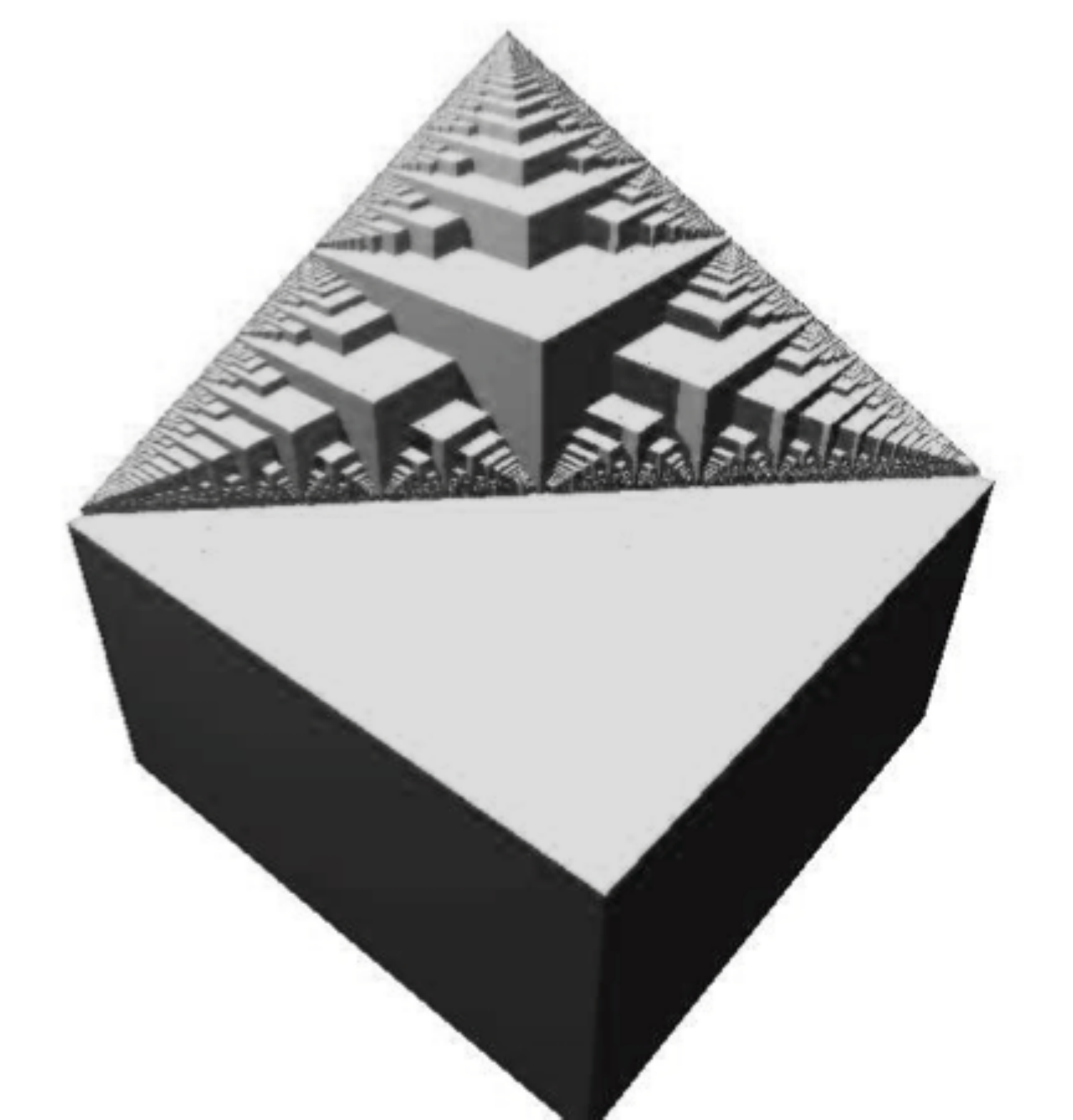} \hfill \caption{The ab-ziggurat (picture credit: Calegari-Walker~\cite[Fig.~1]{Caleg})} \label{ab} \hfill \includegraphics[width=70mm]{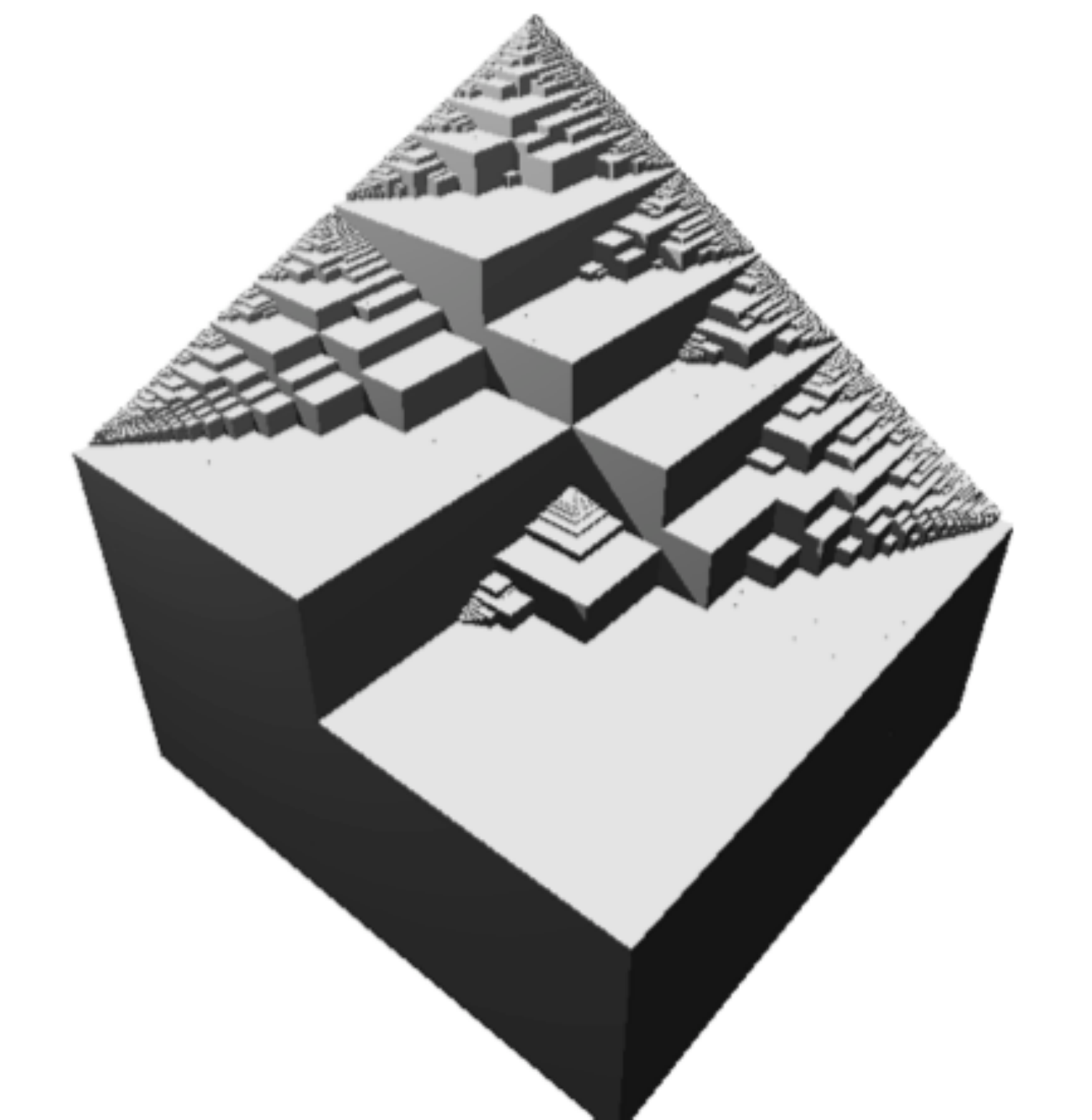} \hfill \caption{The abaab-ziggurat (picture credit:~\cite[Fig.~2]{Caleg})} \label{aabab}
\end{multicols} \end{figure}

To conclude this section, we would like to state some related interesting questions:
\begin{question}
Is there any direct dynamical proof of $T_{1,2}$-invariance of the set $\Delta$ or of its vertical symmetry? Do the transformations $T_{1,2}$ admit any dynamical interpretation? 
\end{question}


\begin{question}\footnote{After this text was finished, we became aware that Subhadip Chowdhury has managed to show ``asymptotic'' projective self-similarity near the ``fringes'' of the unit square for the ziggurats associated to some other positive words . }
What happens for ziggurats associated to other positive products? For instance, the $abaab$-ziggurat also seems to have some fractal, and possibly self-similar nature, see Fig.~\ref{aabab}. 
\end{question}

\begin{question}
Would the ziggurats in higher dimensions (for larger number of different homeomorphisms being multiplied) still look self-similar? (For the Jankins--Neuman higher-dimensional ziggurat, the proof of Theorem~\ref{ss} applies verbatim.) 
\end{question}

\begin{question}
What happens with the ziggurat if the defining Jankins--Neumann-like relation includes \emph{all} the homeomorphisms more than once; for instance, for the relation $abcbac=\id$?
\end{question}


\section{Equivalence}\label{s:equivalence}

This section is devoted to a way of deducing the Jankins--Neumann conjecture from Calegari--Walker's formula~\eqref{formula}. To make such a deduction, let 
\begin{equation}\label{e:zR}
\tz = \{(x,y,z) \mid 0\le z\le R_{ab}(1-x,1-y)-1\}.
\end{equation}
It is easy to see that the Jankins--Neumann conjecture is equivalent to that $\tz=\tilde \tz$. The Calegari--Walker formula allows to represent $\tz$ as a union of parallelepipeds. Indeed, the inequalities $\frac{p_1}{q}\le x$, $\frac{p_2}{q}\le y$ implying $R_{ab}(x,y)\ge \frac{p_1+p_2+1}{q}$ give for the set $\tz$ the representation
\begin{multline}\label{e:tz}
\tz =  \bigcup_{p_1,p_2,q \in \bbN, \atop p_1,p_2\le q} \{(x,y,z) \mid 1-x\le \frac{p_1}{q},\, 1-y\le \frac{p_2}{q} ,  \, z \le \frac{p_1+p_2+1}{q}-1, \, x,y,z\ge 0\}
\\
= \bigcup_{p_1,p_2,q \in \bbN,\atop p_1,p_2 \le q} \Pi\left( 1-\frac{p_1}{q}, 1-\frac{p_2}{q}, \frac{p_1+p_2 + 1}{q}-1\right).
%
\end{multline}
Take $p_1'=q-p_1, p_2'=q-p_2$. Then, 
$$
\Pi\left( 1-\frac{p_1}{q}, 1-\frac{p_2}{q}, \frac{p_1+p_2 + 1}{q}-1\right) =  \Pi\left( \frac{p_1'}{q}, \frac{p_2'}{q}, \frac{q+1-(p_1'+p_2')}{q}\right).
$$
Finally, denoting $p_3':=q-p_1'-p_2'+1$, we get
\begin{equation}\label{e:zigCW}
\tz= \bigcup_{p_1',p_2',p_3',q \in \bbN, \atop p_1'+p_2'+p_3'= q+1} \Pi\left( \frac{p_1'}{q}, \frac{p_2'}{q}, \frac{p_3'}{q}\right).
\end{equation}

To deduce Naimi's theorem (Jankins--Neumann conjecture) from the Calegari--Walker formula, we have to show that $\tz$ and $\tilde \tz$ coincide. The list of parallelepipeds in~\eqref{e:zigCW} contains all the parallelepipeds from the Jankins--Neumann conjecture, but also some others. Let us call a parallelepiped good if it is, up to the permutation of the coordinates, of the form $\Pi(\frac{1}{q}, \frac{p}{q}, \frac{q-p}{q})$ with coprime $p$ and $q$. Then, to prove the coincidence of $\tz$ and $\tilde \tz$ we have to show that any parallelepiped in~\eqref{e:zigCW} that is not a good one, is contained in one of the good ones. For the obvious monotonicity reasons, it suffices to check that its vertex $\left( \frac{p_1'}{q}, \frac{p_2'}{q}, \frac{p_3'}{q}\right)$ is contained there. 

Before proceeding to prove this, we notice (as it was already promised in Remark~\ref{r:min} earlier) that the list of good parallelepipeds is minimal: no good parallelepiped is contained in another one. Indeed, for any of the parallelepipeds listed in~\eqref{e:zigCW}, the sum of any two of three coordinates does not exceed~$1$. As for a good parallelepiped $\Pi\left( \frac{p}{q}, \frac{q-p}{q}, \frac{1}{q}\right)$ the sum of its first two coordinate equals~$1$, it could be contained only in a parallelepiped of the form $\Pi\left( \frac{p'}{q'}, \frac{q'-p'}{q'}, \frac{1}{q'}\right)$ with $\frac{p'}{q'}=\frac{p}{q}$ and $q'<q$. But as $p$ and $q$ are coprime, it is impossible.

Now, let $p_1+p_2+p_3=q+1$, with $p_1,p_2,p_3>1$ (thus, $\Pi\left(\frac{p_1}{q},\frac{p_2}{q},\frac{p_3}{q} \right)$ is not good). We are going then to find a good parallelepiped that contains $\Pi( \frac{p_1}{q}, \frac{p_2}{q}, \frac{p_3}{q})$. Without loss of generality, we can assume $p_1\le p_2\le p_3$. Then, we will be looking for coprime $m<n$ such that
$$
\Pi( \frac{p_1}{q}, \frac{p_2}{q}, \frac{p_3}{q})\subset \Pi( \frac{1}{n}, \frac{m}{n}, \frac{n-m}{n}),
$$ 
or, what is the same,
$$
\frac{p_1}{q} \le \frac{1}{n}, \quad \frac{p_2}{q}\le \frac{m}{n}, \quad \frac{p_3}{q}\le \frac{n-m}{n}.
$$
Or, equivalently,
\begin{equation}\label{eq:mn}
n\le \left[\frac{p_1}{q}\right], \quad \frac{p_2}{q}\le \frac{m}{n} \le 1- \frac{p_3}{q}.
\end{equation}

Denote $N:=\left[\frac{p_1}{q}\right]$. Then, the desired~\eqref{eq:mn} could be reformulated in the following way: we want to prove that there is a fraction $\frac{m}{n}$ of denominator less than $N$ that belongs to the interval $[\frac{p_2}{q} , \frac{q-p_3}{q}]$. Hence it is natural to consider \emph{Farey sequence} of order~$N$.

Recall that the Farey sequence $F_N$ of order $N$ is the sequence of completely reduced fractions between~$0$ and~$1$ which have denominators less than or equal to~$N$, arranged by increasing. In the proof of the equivalence of two ziggurats we will use one of its properties, namely, for any two consecutive fractions $\frac{a}{b}<\frac{c}{d}$ in any of the Farey sequences, their denominators are coprime, and moreover, $\frac{c}{d}-\frac{a}{b}=\frac{1}{bd}.$ As for details, we refer the reader to ~\cite{F2} (as well as to the original historical papers~\cite{F0,F1}). 

Suppose there is no $\frac{m}{n} \in F_N$ between $\frac{p_2}{q}$ and $\frac{q-p_3}{q}$. Then, take the two consecutive fractions $\frac{a}{b}, \frac{c}{d} \in F_N$ such that $\frac{a}{b}<\frac{p_2}{q}\le\frac{q- p_3}{q}<\frac{c}{d}$. In this case,
$$
\frac{p_2}{q}-\frac{a}{b}\ge \frac{1}{bd}, \quad \frac{c}{d}-\frac{q - p_3}{q}\ge \frac{1}{dq}.
$$
Thus,
\begin{multline} \label{eq:bdq}
\frac{q - p_3}{q} - \frac{p_2}{q}=\left(\frac{c}{d}-\frac{a}{b}\right) - \left(\frac{c}{d}-\frac{q - p_3}{q}\right) - \left(\frac{p_2}{q}-\frac{a}{b}\right) \le\\
\le \frac{1}{bd} - \frac{1}{bq} - \frac{1}{dq} = \frac{q-(b+d)}{bd\cdot q}.
\end{multline}

If both $b,d \ge 2$, then $bd \ge b+d\ge N+1 > \frac{q}{p_1}$, we see that the right hand side of~\eqref{eq:bdq} can be estimated as 
$$
\frac{q-(b+d)}{bd\cdot q} < \frac{q-\frac{q}{p_1}}{\frac{q}{p_1}\cdot q} = \frac{p_1-1}{q}.
$$
Thus, $\frac{q - p_3}{q} - \frac{p_2}{q}< \frac{p_1-1}{q}$, contradicting the equation $p_1+p_2+p_3=q+1$.

Finally we have to consider the possibilities $b=1$ or $d=1$. Though, as $\frac{p_2}{q} + \frac{p_3}{q} < 1$, we have $\frac{p_2}{q}<\frac{1}{2}$ and hence $d \ge 2$. Finally, if $b=1$ and thus $\frac{a}{b} = 0, \frac{c}{d}=\frac{1}{N}$, we have 
$$
1-\frac{p_3}{q}=\frac{p_2+p_1-1}{q}\ge \frac{2p_1-1}{q} \ge \frac{3}{2} \cdot \frac{p_1}{q}\ge \frac{3}{2} \cdot \frac{1}{N+1}.
$$ 
Then, the inequality $1-\frac{p_3}{q}<\frac{c}{d}=\frac{1}{N}$ is impossible, as it would imply $\frac{3}{2(N+1)} < \frac{1}{N}$ and hence $N< 2$. Though, $\frac{p_1}{q}\le \frac{1}{2}$ and hence $N\ge 2$.



\begin{figure}[h] \begin{multicols}{2}
\hfill \includegraphics[width=60mm]{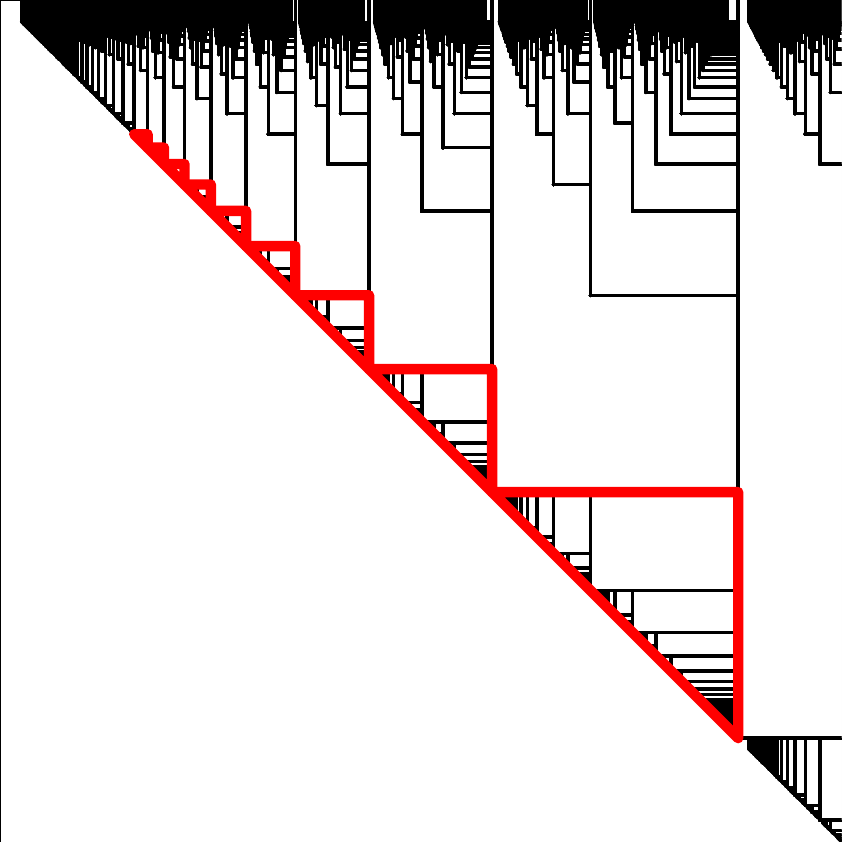} \hfill \caption{The $\Tr_n$ triangles marked on the view on the ab-ziggurat from the top} \label{f:t-red} \hfill \includegraphics[width=60mm]{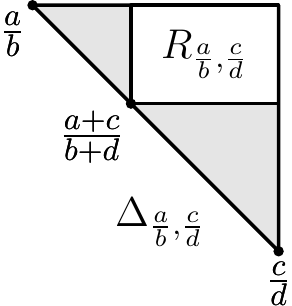} \hfill \caption{Triangle $\Tr_{\frac{a}{b},\frac{c}{d}}$ and two its descendants $\Tr_{\frac{a}{b},\frac{a+c}{b+d}}$ and $\Tr_{\frac{a+c}{b+d},\frac{c}{d}}$.} \label{f:induction}
\end{multicols} \end{figure}

These contradictions conclude the proof of the equivalence. 


Finally, we would like to note a beautiful connection of ziggurat and Farey series. Consider the set of triangles on the plane $Oxy$, enclosed by lines $x+y=1$, $x=1/n$ and $y=n/(n+1)$, where $n=2,3,\dots$. Let us denote $n$th such triangle by~$\Tr_n$:
\[
\Tr_n := \{(x,y)\in [0;1)^2 \mid x+y>1,\, x< 1/n,\, y< n/(n+1) \}  
\]

These triangles naturally appear on the view from the top of the ab-ziggurat (Jankins-Neumann up to a linear transformation): see Fig.~\ref{proj} and~\ref{f:t-red}; on the latter they are marked by bold red lines. It is easy to see from the Jankins--Neumann's description, that all the vertices project into the diagonal $x+y = 1$ or outside these triangles. 

Now, looking on a Fig.~\ref{f:t-red}, we see that each of these triangles is decomposed into a rectangle and two more triangles, each of these triangles decomposes into a rectangle and two more, etc. Each new rectangle corresponds to a vertex $(x,1-x)$ with rational $x$, and it is interesting to study the order, in which these vertices appear, and also prove it formally. To do so, note that any of these triangles is of the form
$$\Tr_{\textstyle{\frac{a}{b},\frac{c}{d}}}:=\{(x,y)\in [0;1)^2 \mid x+y>1,\, x< \frac{c}{d},\, y< 1-\frac{a}{b} \} $$
for some $\frac{a}{b}<\frac{c}{d}$.

Let us show that the fractions $\frac{a}{b}$ and $\frac{c}{d}$ are adjacent in one of the Farey series, and the rectangle that subdivides this triangle starts from a point with the abscise that is the \emph{mediant} (``freshman sum'')  $\frac{a+c}{b+d}$ of $\frac{a}{b}$ and~$\frac{c}{d}$; see Fig.~\ref{f:induction}. Indeed, the subdivisions that we are observing are given by removing of the slices $\{(x,y) \mid R_{ab}(x,y)-1\ge r\}$ for some $r>0$. As $R_{ab}(\frac{p}{q},1-\frac{p}{q})-1=\frac{1}{q}$ for coprime $p,q$, the vertices appearing in any such slice are exactly those whose abscises belong to $[\frac{1}{r}]$-th Farey sequence. And when the triangle $\Tr_{\textstyle{\frac{a}{b},\frac{c}{d}}}$ is subdivided, it corresponds to the subdivision of an interval $[\frac{a}{b},\frac{c}{d}]$ of a Farey sequence; {as it is well-known (see~\cite{F2}), it is the mediant that subdivides it}.

Thus, the Farey sequences are not only a good tool for the study of the ziggurat, but they appear in this study in a natural way. We conclude this section by noticing an interesting consequence of the above arguments: we see that two fractions are neighbours in some Farey series if and only if the rectangles, starting at the corresponding points, have common segment.

\section*{Acknowledgements}

Author is exceedingly grateful to Danny Calegari, Walter Neumann, Alden Walker and Subhadip Chowdhury for their comments and cooperation and to Victor Kleptsyn for discussions and assistance.

\end{document}